\documentclass[letterpaper, 10pt, conference]{ieeeconf}

\IEEEoverridecommandlockouts                              
\usepackage{amsmath,amsthm,enumerate,pdfsync,asymptote}  
\usepackage{setspace} 
\usepackage{graphicx,subfig}
\usepackage{amssymb}
\usepackage{setspace}
\usepackage{graphics,framed}
\usepackage{pgf}
\usepackage{tikz}
\usepackage{pgfplots} 

\newtheorem{Definition}{Definition}
\newtheorem{Example}{Example}
\newtheorem{Proposition}{Proposition}
\newtheorem{Theorem}{Theorem}

\newtheorem{Corollary}{Corollary}	
\newtheorem{Remark}{Remark}
\newcommand{\R}{\mathbb{R}}

\usetikzlibrary{calc,fadings,decorations.pathreplacing}

\newcommand{\scone}[3]{%
\begin{scope}[rotate=#3,xshift=#1,yshift=#2]
\def\mypath{ (.-.2,0) -- +(-.2,.8) arc (180:0:.8) -- +(-.2,-.8) arc (-180:0:.2) } 
\fill [gray] mypath;
}

\tikzset{%
  >=latex, 
  inner sep=0pt,%
  outer sep=2pt,%
  mark coordinate/.style={inner sep=0pt,outer sep=0pt,minimum size=3pt,
    fill=black,circle}%
}

\makeatletter
\newsavebox{\sfe@box}
{\color@endgroup\egroup\subfloat[\sfe@caption]%
{\usebox{\sfe@box}}}
\makeatother

\title{\LARGE \bf
Mathematical aspects of decentralized control of formations in the plane}

\author{M.-A. Belabbas 
\thanks{M.-A. Belabbas is with the School of Engineering and Applied Sciences, Harvard
University, Cambridge, MA 02138 {\tt\small belabbas@seas.harvard.edu}}%
}

\begin{document}

\maketitle

\begin{abstract}                
In formation control, an ensemble of autonomous agents  is required to stabilize at a given configuration in the plane, doing so while agents are allowed to observe only a subset of the ensemble. As such, formation control provides a rich class of problems for decentralized control methods and techniques. Additionally, it can be used to model a wide variety of scenarios where decentralization is a main characteristic. We  introduce here  some mathematical background necessary to address questions of stability in decentralized control in general and formation control in particular. This background includes an extension of the notion of global stability to systems evolving on manifolds and a notion of robustness of feedback control for nonlinear systems. We then formally introduce the class of formation control problems, and summarize known results.
\end{abstract}

\section{Introduction}

We present here some concepts and definitions related to the study of decentralized  and multi-agent systems in general and to formation control in particular. 

We start with the introduction of \emph{type-A stability}. It has been known since at least Poincar\'e that the topology of the manifold on which a system evolves strongly affects the type of dynamics that are possible. In particular, global stability as it is defined for systems on vector spaces is often trivially impossible when the manifold is not a vector space. We propose here a definition that is meaningful for systems evolving on manifold and  captures the practical  benefits of global stabilization. 

The second definition is the one of \emph{robustness}. When solving a control design problem, one is faced with finding a control $u^*$, belonging to admissible set of control $\mathcal U$, that achieves a given objective, e.g. stabilization around a given configuration.  In real-world applications, one is of course often confronted to errors in modelling, noise in the inputs or in the observations, or other sources of uncertainty that may make a control law designed for an ideal situation fail. We introduce below a notion of robustness, akin to the one of linear systems theory, that allows us to handle such situation.

The  introduction of robustness comes with an unexpected benefit: a simplification of the design problem.  Indeed, if there exist a control law that achieves a given objective non-robustly, this control would be quite difficult to find. In practical terms, robustness allows us to confine our search to the  jet-space of lowest possible order~\cite{belabbasSICOpart2} (we give a brief introduction to jet spaces in the appendix). 

In Section~\ref{sec:formation}, we formally introduce the class of formation control problems. Our approach, which puts at the center \emph{configurations of points}, and allows us to understand the role of rigidity theory as a way to decentralize the \emph{global objective}: in the language of the companion paper~\cite{belabbas_icca2011_knowns}, rigidity has to do with the $\delta$ functions, and using it to define the information flow is thus in many ways unnatural.

We conclude by  summarizing what is known about formation control and about the so-called  2-cycles formation~\cite{belabbascdc11sub2,cao2010festschrift}. We mentioned in~\cite{belabbas_icca2011_knowns} that a major issue in decentralization is the existence of nontrivial loops of information---that is loop of informations that the system cannot by-pass. The 2-cycles is the simplest formation that exhibits two nontrivial loops in its information flow graph. These information loops are the main source of difficulty in the analysis of the system~\cite{cao2010festschrift}.
 
\section{Type-A stability}

Many natural and engineering systems are described by a differential equation evolving on a manifold $M$, by opposition to a flat space or vector space. For example, the orientation of a rigid body in space is described by a point in the Lie group $SO(3)$~\cite{bloch_nonholonbook_03}; another example arise in formation control: we have shown~\cite{belabbascdc11sub1} that, due to the invariance of the   system under rotations and translations, the state-space of $n$ autonomous agents in the plane is given by the manifold $\mathbb CP(n-2)\times (0,\infty)$. 

When the system evolves on a manifold, global notions such as global stabilization need to be adjusted to remain relevant. This is the issue addressed by type-A stability.

Consider the control system \begin{equation}
\label{eq:sysu1}\dot x = f(x,u(x))
\end{equation} where $x \in {M}$, a smooth manifold, and all functions are assumed smooth.

According to elementary results in Morse theory~\cite{milnormorse}, if  the manifold $M$ possesses non-trivial homology groups~\cite{warner83}, the system~\eqref{eq:sysu1} cannot be globally stable in the usual sense:  there is no continuous $u$ such that~\eqref{eq:sysu1} has a \emph{unique} equilibrium.

From a practical standpoint, however, if one could make  one equilibrium stable, and all other equilibria either saddles or unstable, the system would behave as if it were globally stable. Indeed, a vanishingly small perturbation would ensure that the system, if at a saddle or unstable equilibrium,  evolves to the unique stable equilibrium. We formalize and elaborate on this observation. 

Let $\mathcal{E}_d$ be a finite subset of $M$ containing  configurations that we would like to stabilize via feedback.   We are thus interested in the design of a smooth feedback control $u(x)$ that will stabilize the system to any point $x_0\in \mathcal{E}_d$. We call these points the \emph{design targets} or \emph{design equilibria}:

$$\mathcal{E}_d = \lbrace x_0 \in M \mbox{ s.t. } x_0 \mbox{ is a design equilibrium} \rbrace $$

Let $$\mathcal{E}  = \lbrace x_0 \in M \mbox{ s.t. } f(x_0,u(x_0)) = 0 \rbrace, $$ the set of equilibria of~\eqref{eq:sysu1}. We  assume that $\mathcal{E}$ is {finite}. 

As explained above, when the system evolves on a non-trivial manifold, the Morse inequalities make it unreasonable  to expect that there exists a control $u(x)$ that makes the design equilibria the \emph{only} equilibria of the system, i.e. a control such that $\mathcal{E}_d= \mathcal{E}$.  We call the additional equilibria, that are introduced by the non-trivial topology of the space, \emph{ancillary equilibria}: $$\mathcal{E}_a = \mathcal{E}-\mathcal{E}_d.$$

Let us assume for the time being that the linearization of the system at an equilibrium has no  eigenvalues with zero real part. We decompose the set $\mathcal{E}$ into \emph{stable} equilibria, by which we mean  equilibria such that \emph{all the eigenvalues} of the linearized system have a negative real part, and \emph{unstable equilibria}, where \emph{at least one eigenvalue} of the linearization has a positive real part. Observe that under this definition,   saddle points are considered unstable.

In summary: $$\mathcal{E} = \mathcal{E}_s \cup \mathcal{E}_u$$ where
$$\mathcal{E}_s = \lbrace x_0 \in \mathcal{E}\ |\ x_0 \mbox{ is stable} \rbrace$$ and
$$\mathcal{E}_u = \lbrace x_0 \in \mathcal{E}\  |\ x_0 \mbox{ is unstable} \rbrace.$$

With these notions in mind, we introduce the following definition:

\begin{Definition}
Consider the smooth control system $\dot x = f(x,u(x))$ where $x \in M$ and the set $\mathcal{E}$ of equilibria of the system is finite. Let $\mathcal{E}_d \subset M$ be a finite set. We say that $\mathcal{E}_d$ is
\begin{enumerate}
\item \emph{feasible} if we can choose a smooth $u(x)$ such that $\mathcal{E}_d \cap \mathcal{E} \neq \varnothing$.
\item \emph{type-A stable} if we can choose a smooth $u(x)$ such that $\mathcal{E}_s \subset \mathcal{E}_d$.
\item \emph{strongly type-A stable} if we can choose a smooth $u(x)$ such that $\mathcal{E}_s = \mathcal{E}_d$.
\end{enumerate}
When the set $\mathcal{E}_d$ is clear from the context, we say that the system is feasible or type-A stable. 
\end{Definition}

This definition extends trivially to systems depending on a parameter. The set $\mathcal{E}_d$ is feasible if we can choose $u(x)$ such that \emph{at least one} equilibrium  of the system is a design target. It is said to be \emph{type-A} stable if the system  stabilizes to $\mathcal{E}_d$ with probability one for any randomly chosen initial conditions on $M$.  It is \emph{strongly type-A} stable if it is type-A stable and moreover all elements of $\mathcal{E}_d$  are stable equilibria. The usual notion of global stability is a particular instance of type-A stability; indeed, it  corresponds to having $u(x)$ such that  $\mathcal{E}_d=\mathcal{E}=\mathcal{E}_s$. 

Looking at  the contrapositive of this definition,  a system  is \emph{not type-A stable} if there exists a set of initial conditions,  of strictly positive measure, that lead to an ancillary equilibrium.  We observe that type-A stability is  a global stability notion; in particular,  if one can choose $u$ such that all design equilibria are locally stable, but if this choice forces the appearance of other, undesired equilibria which are also locally stable, the system is not type-A stable. The example below illustrate these notions.

\begin{Example}
Consider a system $$\dot x = x(1-kx^2)$$ where $k \in \R$ is a feedback parameter to be chosen by the user. We show that any $\mathcal{E}_d \subset (0,\infty)$ is not type-A stable. We first observe that the system has an equilibrium at $0$ and two equilibria at $x = \pm \sqrt{1/k}$ if $k >0$.  The system is thus feasible for any $\mathcal{E}_d \subset \R$.  The Jacobian of the system is $1$ at $x=0$ and $-2$ at $x=\pm \sqrt{1/k}$. For $k>0$, the above says that $$\mathcal{E}= \lbrace 0, \pm \sqrt{1/k} \rbrace= \underbrace{\lbrace \sqrt{1/k} \rbrace}_{\mathcal{E}_d} \cup \underbrace{ \lbrace 0, -\sqrt{1/k} \rbrace}_{\mathcal{E}_a}.$$  From the linearization of the system, we have that $$\mathcal{E}_s = \lbrace \pm \sqrt{1/k}\rbrace \mbox{ and }\mathcal{E}_u = \lbrace 0 \rbrace.$$ We conclude that $\mathcal{E}_s \nsubseteq \mathcal{E}_d$ and the system is not type-A stable. 
\end{Example}

\section{Robustness}\label{sec:sing}
We introduce here a definition of robustness for nonlinear systems. We start by discussing the well-established concept of generic elements, on which our definition of robustness is based.

Informally speaking, a property of elements of a topological space is said to be \emph{generic} if it is shared by \emph{almost all} elements of the set. 
\begin{Definition} A property $\mathcal{P}$ is \emph{generic} for a topological space $S$ if it is true on an everywhere dense intersection of open sets of $S$.
\end{Definition}

Everywhere dense intersections of open sets are sometimes  called \emph{residual} sets~\cite{arnold_bifurcation}. In general, asking for a given property to be generic is a rather strong requirement, and oftentimes it is enough to show that a given property is true on an open set of parameters, initial conditions, etc. We define

\begin{Definition} An element $u$ of a topological space $S$ satisfies the  property $\mathcal{P}$  \emph{robustly} if $\mathcal P $ is true for all $u'$  in a neighborhood of $u$ in $S$. A property $\mathcal P$ is robust if there exists a robust $u$ which satisfies the property.
\end{Definition}

In practical terms, if a property satisfied only at \emph{non-robust} $u$'s, then it will likely fail to be satisfied under the slightest  error in modelling or measurement. 

\begin{Remark}
We emphasize that when we seek a robust control law $u(x)$ for stabilization, we seek a control law such that  the equilibrium that is to be stabilized remains stable under small perturbations in $u(x)$. The equilibrium, however, may move in the state space. For example, assume that the system $$\dot x = u(x)$$ has the origin as a stable equilibrium. If for all $g(x)$ in an appropriate set of perturbations, the system $$\dot x=u(x) + \varepsilon g(x)$$ has a stable equilibrium at a point $z(\varepsilon)$ near the origin, then the control law $u(x)$ is robust. If, on the contrary, the equilibrium disappears or becomes unstable, then $u(x)$ is not robust.
\end{Remark}

If $\rceil \mathcal P$, the negation of $\mathcal P$, is generic, then there is no robust $u$ that satisfies $\mathcal P$. Indeed, if $\rceil \mathcal P$ is generic, then $\mathcal{P}$ is verified on at most a nowhere dense closed set. In particular, $\mathcal P$ is not verified on an open set. The main tool to handle genericity are jet spaces and Thom transversality theorem. We will use the results in some parts below and refer the reader to the appendix and to~\cite{belabbasSICOpart2} for more information.

\section{Formation control}\label{sec:formation}

We present here the class of formation control problems in the plane. This class provides a rich set of examples and models for decentralized control.

We begin with some preliminaries. We call a \emph{configuration of $n$ points in the plane} an equivalence class, under rotation and translation, of $n$ points in $\R^2$, see Figure~\ref{fig:xbar8} for an example. We have shown in~\cite{belabbascdc11sub1} that the space of such normalized equivalence classes was a complex projective space.

Let $G=(V,E)$ be a \emph{graph} with $n$ vertices --- that is $V = \lbrace x_1 ,x_2,\ldots,x_n \rbrace$  is an ordered set of vertices and $E \subset V \times V$ is a set of edges. The graph is said to be \emph{directed} if $(i,j) \in E$ does not imply that  $(j,i) \in E$.  We let $|E| = m $ be the cardinality of $E$. We call the \emph{outvalence} of a vertex the number of edges originating from this vertex and the \emph{invalence} the number of incoming edges.

\subsection{Rigidity }
We briefly cover the fundamentals of rigidity and establish the relevant notation. We refer the reader to~\cite{belabbascdc11sub1,belabbasSICOpart1} for a more detailed presentation. 
We call a \emph{framework} an embedding of a graph in $\R^2$ endowed with the usual Euclidean distance, i.e. given $G=(V, E)$, a framework $p$ attached to a graph $G$ is a mapping \begin{equation*} p: V \rightarrow \R^2.
\end{equation*}

By abuse of notation, we  write $x_i$ for $p(x_i)$. We define the  \emph{distance function} $\delta$ of a framework with $n$ vertices as \begin{multline*}
\delta(p): \R^{2n} \rightarrow \R_+^{n(n-1)/2}:(x_1, \ldots, x_n)  \rightarrow \frac{1}{2}\left[ \|x_1-x_2\|^2, \right. \\ \left. \ldots,  \|x_1-x_n\|^2,  \| x_2-x_3\|^2, \ldots,  \|x_{n-1}-x_n\|^2  \right],
\end{multline*} where $\R^+ = [0, \infty) $. We denote by $\delta(p)|_E$ the restriction of the range of $\delta$ to edges in $E$.

For a graph $G$ with $m$ edges, we define \begin{multline*}\mathcal{L}=\left\lbrace d=(d_1,\ldots,d_m ) \in \R^m_+ \mbox{ for which }  \right.\\ \left. \exists p \mbox{ with } \delta(p(V))|_E={d} \right\rbrace,\end{multline*} where the square root of $d$ is taken entry-wise. Properties of this set and its relations to the number of ancillary equilibria are discussed in~\cite{belabbasSICOpart1,belabbascdc11sub1}.

The \emph{rigidity matrix} of the framework is the Jacobian   $\frac{\partial \delta}{\partial x}$  restricted to the edges in $E$. We denote it by $\frac{\partial \delta}{\partial x}|_E$. 

\begin{Definition}[Rigidity]
\begin{enumerate}
\item A framework is said to be infinitesimally rigid if there are no vanishingly small motions of the vertices, except for rotations and translations,  that keep the edge-length constraints on the framework satisfied. This translates into~\cite{graver93}
$$\operatorname{rank} (\frac{\partial \delta}{\partial x}|_{E})=2n-3.$$
\item  A framework attached to a graph $G$ is said to be \emph{rigid} if there are no motions of the vertices that keep the edge lengths constraints satisfied and \emph{minimally rigid} if all the edges of the graph are necessary for rigidity. 
\end{enumerate}
\end{Definition}

\subsection{Formation control: definition and open problems}

We present here the definition of formation control problems. We build the problem around configurations of points in the plane, by opposition to distances between vertices; this allows us to understand the role of rigidity in formation control as a tool to address the distribution of the global objective or---with the notation of the companion paper~\cite{belabbas_icca2011_knowns}---as a tool to determine which $\delta_i$ are sufficient. This point of view also makes clear that there is no reason to assume that the functions $h_i$ describing the information flow should be given by a rigid graph. In fact, this  overload of $G$ in formation control is a limiting factor as is illustrated in Section~\ref{sec:2cycles}.

We let $x \in \R^{2n}$ contain the positions of all the agents in the formation and consider   general dynamical models of the form. 
\begin{equation}
\dot x = \sum_{i=1}^n \sum_{j=1}^{n_i} u_{ij}(\delta_i(\mu);h_i(x)) g_{ij}(x),
\end{equation} where $u_{ij}$ is a real function, $g_{ij}$ are smooth vector fields and $\delta, h$ are smooth vector valued functions. We analyzed this model in detail in the companion paper~\cite{belabbas_icca2011_knowns}

\subsubsection{Configurations of n-points}

The objective in formation control is a parametric one. Let $P$ be the space of configurations of $n$ points in $\R^2$, up to rigid transformations of the plane: i.e. a point in $P$ is an equivalence class of points in $\R^{2n}$.  For our purpose here, it is enough to describe a configuration of $n$ points in the place by an element of $\R^{2{n-1}}$, where we use the translational degree of freedom to set the first point at the origin in $\R^2$. We represent a design formation by a vector $$\mu \in \R^{2(n-1)}=[\bar x_2, \ldots, \bar x_n ], \bar x_i \in \R^2$$ as illustrated in  in Figure~\ref{fig:xbar8}. The vector $\mu$ is thus a representative of the equivalence class of points obtained via rotation of the $\bar x_i$.

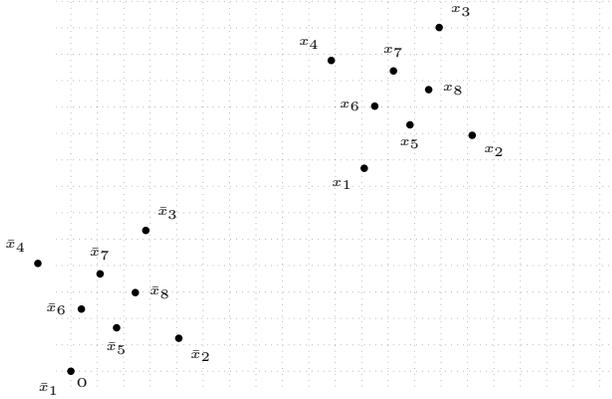
\begin{figure}
\begin{center}
\begin{tikzpicture}
\draw[help lines, step=10pt,dotted](-.2,-.2) grid (7.2,5.2);
\begin{scope}[scale=1.5,xshift=2.6cm, yshift=1.8cm, rotate=17]
\node [fill=black,circle, inner sep=1pt,label=225:{\tiny $x_1$}] (1) at (0,0) {};
\node [fill=black,circle, inner sep=1pt,label=-45:{\tiny $x_2$}] (2) at (1,0) {};
\node [fill=black,circle, inner sep=1pt,label=135:{\tiny $x_4$}] (4) at (0,1) {};
\node [fill=black,circle, inner sep=1pt,label=45:{\tiny $x_3$}] (3) at (1,1) {};

\node [fill=black,circle, inner sep=1pt,label=-90:{\tiny $x_5$}] (1) at (0.5,0.25) {};
\node [fill=black,circle, inner sep=1pt,label=180:{\tiny $x_6$}] (2) at (.25,.5) {};
\node [fill=black,circle, inner sep=1pt,label=90:{\tiny $x_7$}] (4) at (.5,.75) {};
\node [fill=black,circle, inner sep=1pt,label=0:{\tiny $x_8$}] (3) at (.75,.5) {};
\end{scope}
\begin{scope}[scale=1.5,rotate=17]
\node [fill=black,circle, inner sep=1pt,label=225:{\tiny $\bar x_1$}] (1) at (0,0) {};
\node [fill=black,circle, inner sep=1pt,label=-45:{\tiny $\bar x_2$}] (2) at (1,0) {};
\node [fill=black,circle, inner sep=1pt,label=135:{\tiny $\bar x_4$}] (4) at (0,1) {};
\node [fill=black,circle, inner sep=1pt,label=45:{\tiny $\bar x_3$}] (3) at (1,1) {};

\node [fill=black,circle, inner sep=1pt,label=-90:{\tiny $\bar x_5$}] (1) at (0.5,0.25) {};
\node [fill=black,circle, inner sep=1pt,label=180:{\tiny $\bar x_6$}] (2) at (.25,.5) {};
\node [fill=black,circle, inner sep=1pt,label=90:{\tiny $\bar x_7$}] (4) at (.5,.75) {};
\node [fill=black,circle, inner sep=1pt,label=0:{\tiny $\bar  x_8$}] (3) at (.75,.5) {};
\end{scope}
\node (or) at (0.15,-.15) {\tiny O};
\end{tikzpicture}
\end{center}
\caption{Configurations of $n$ points in the plane, up to translations, can be represented by a configuration with $x_1$ at the origin. We use the notation $\bar x_i$ for the coordinates of the points after translation.}\label{fig:xbar8}
\end{figure}

In order to represent the decentralized structure of a formation control problem, we introduce two graphs: the $\delta$-graph and the $h$-graph, representing respectively the information a given agent has about the global objective of the formation ($\delta$-graph) and about the state of the formation ($h$-graph).

In more detail, to each agent with position $x_i \in \R^2$ we associate a vertex $i$ in $V$. We let $G$ be a graph with vertex set $V$ and set of edges $E$. The edges define the decentralized structure as follows: assume that vertex $i$ has outvalence $k$ and  the edges $(i,{l_1}), \ldots, (i, {l_k}) \in E_h$ are leaving from  vertex $i$.  The $h$-graph $G_h=(V,E_h)$ defines the functions $h_i(x)$ according to:

\begin{enumerate}
\item[1:] Range only information: in this case, an agent is only able to measure its distance to its various neighbors.
$$h^R_i(x): \R^{2n} \rightarrow \R^{k}: x \rightarrow (\|x_{l_1}-x_i\|, \ldots,\|x_{l_k}-x_i \|).$$

\item[2:] Relative position information: in this case, agent $i$ can measure the \emph{relative positions of its neighbors}. Using some simple trigonometric rules, it is easy to see that in order to reconstruct the relative positions of its neighbors (i.e. their position relative to $x_i$, up to rotation), it is sufficient for agent $i$ to have the distances to the agents as well as the inner products $(x_{l_1}-x_i)^T(x_{l_j}-x_i)$, $j=2\ldots k$. We have
\begin{multline*}h_i(x)= (h^R_i(x), (x_{l_1}-x_i)^T(x_{l_2}-x_i)^T, \ldots, \\(x_{l_k}-x_i)^T(x_{l_1}-x_i)).\end{multline*}
\end{enumerate}
A \emph{formation} is an ensemble of agents together with an $h$-graph.

\begin{Remark} The h-graph, which is proper to formation control, is related to the information flow graph defined~\cite{belabbas_icca2011_knowns}.
\end{Remark}
The function $\delta_i$ are similarly described by a graph $G_\delta=(V,E_\delta)$. Assume that vertex $i$ has outvalence $k$ and  the edges $(i,{l_1}), \ldots, (i, {l_k}) \in E_\delta$ are leaving from  vertex $i$.  We have

\begin{enumerate}
\item[1:] Range only information: an agent only knows about the distance at which it needs to stabilize from its neighbors.
$$\delta_i(\mu) = (\|\bar x_{l_1}-\bar x_i\|^2, \ldots, \|\bar x_{l_k}-\bar x_i\|^2).$$
\item[2:] Range and angle: in this case, the agents also knows the relative position at which its neighbors are in the target framework:
\begin{multline*}\delta_i(\mu) =(\delta^R_i(\mu), (\bar x_{l_1}-\bar x_i)^T(\bar x_{l_2}-\bar x_i)^T, \ldots, \\(\bar x_{l_k}-\bar x_i)^T(\bar x_{l_1}-\bar x_i)).\end{multline*}

\end{enumerate}

\begin{Example}
Consider a formation control problem where we require the agents to stabilize at the configuration of points described in Figure~\ref{fig:dhc}. 

The $h$-graph of Figure~\ref{fig:dhh1} corresponds to the observation functions, assuming the relative position case:
\begin{eqnarray*}
h_1(x) &=&  \|x_2-x_1\| \\
h_2(x) &=&  \left(\|x_3-x_2\|,\|x_4-x_2\|,(x_3-x_2)^T(x_4-x_2) \right)  \\
h_3(x) &=&  \|x_3-x_1\| \\
h_4(x) &=&  \|x_4-x_1\| \\
h_5(x) &=&  \left(\|x_2-x_5\|,\|x_3-x_5\|,(x_2-x_5)^T(x_3-x_5) \right) .
\end{eqnarray*}
Similarly,  the $h$-graph of Figure~\ref{fig:dhh2} corresponds to
\begin{eqnarray*}
h_1(x) &=&  \|x_2-x_1\| \\
h_2(x) &=&  \left(\|x_3-x_2\|,\|x_4-x_2\|,(x_3-x_2)^T(x_4-x_2) \right)  \\
h_3(x) &=&  \|x_3-x_5\| \\
h_4(x) &=&  \|x_4-x_5\| \\
h_5(x) &=&  \left(\|x_2-x_5\|,\|x_1-x_5\|,(x_1-x_5)^T(x_2-x_5) \right) .
\end{eqnarray*}
The $\delta$-graph of Figure~\ref{fig:dhd1} corresponds to the functions $\delta_i$ given by 
\begin{eqnarray*}
\delta_1(\mu) &=&  \|\bar x_2-\bar x_1\| \\
\delta_2(\mu) &=&  \left(\|\bar x_3-\bar x_2\|,\|\bar x_4-\bar x_2\| \right)  \\
\delta_3(\mu) &=&  \|\bar x_3-\bar x_5\| \\
\delta_4(\mu) &=&  \|\bar x_4-\bar x_5\| \\
\delta_5(\mu) &=&  \left(\|\bar x_2-\bar x_5\|,\|\bar x_1-\bar x_5\| \right) 
\end{eqnarray*}
in the case of distance only information. In the case of relative position information, $\delta_2$ and $\delta_4$ would also contain the inner products of the appropriate $\bar x_i$. The $\delta$-graph of Figure~\ref{fig:dhd2} corresponds to letting every agent know its distance to all other agents in the case  of range only information, and letting $\delta_i(\mu)=\mu$ for all $i$ in the case of relative position.

\end{Example}

\begin{figure}
\begin{center}
\subfloat[]{
\begin{tikzpicture}[scale=1.2]
\draw[help lines, step=10pt,dotted](-.2,-.2) grid (2.2,1.9);
\node [fill=black,circle, inner sep=1pt,label=225:{\tiny $\bar x_1$}] (1) at (0,0) {};
\node [fill=black,circle, inner sep=1pt,label=-45:{\tiny $\bar x_2$}] (2) at (2,0) {};
\node [fill=black,circle, inner sep=1pt,label=180:{\tiny $\bar x_4$}] (4) at (0,1) {};
\node [fill=black,circle, inner sep=1pt,label=45:{\tiny $\bar x_3$}] (3) at (1,1.5) {};
\node [fill=black,circle, inner sep=1pt,label=90:{\tiny $\bar  x_5$}] (5) at (.5,1) {};
\end{tikzpicture}\label{fig:dhc}
}\\
\subfloat[]{
\begin{tikzpicture}[scale=1.4]
\node [fill=black,circle, inner sep=1pt,label=225:{\tiny $x_1$}] (1) at (0,0) {};
\node [fill=black,circle, inner sep=1pt,label=-45:{\tiny $x_2$}] (2) at (1,0) {};
\node [fill=black,circle, inner sep=1pt,label=180:{\tiny $x_4$}] (4) at (-.15,1.1) {};
\node [fill=black,circle, inner sep=1pt,label=0:{\tiny $x_3$}] (3) at (1.15,1) {};
\node [fill=black,circle, inner sep=1pt,label=90:{\tiny $x_5$}] (5) at (.5,1.58) {};

\draw [,-stealth ] (1) -- (2) node [ midway,above,  sloped, blue] {}; 
\draw [,-stealth ] (3) -- (1) node [ midway,above,  sloped, blue] {}; 
\draw [,-stealth ] (2) -- (3) node [ midway,above,  sloped, blue] {}; 
\draw [,-stealth ] (2) -- (4) node [ midway,above,  sloped, blue] {}; 
\draw [,-stealth ] (5) -- (2) node [ midway,above,  sloped, blue] {}; 
\draw [,-stealth ] (5) -- (3) node [ midway,above,  sloped, blue] {}; 
\draw [,-stealth ] (4) -- (1) node [ midway,above,  sloped, blue] {}; 

\end{tikzpicture}\label{fig:dhh1}
}\quad
\subfloat[]{
\begin{tikzpicture}[scale=1.4]
\node [fill=black,circle, inner sep=1pt,label=225:{\tiny $x_1$}] (1) at (0,0) {};
\node [fill=black,circle, inner sep=1pt,label=-45:{\tiny $x_2$}] (2) at (1,0) {};
\node [fill=black,circle, inner sep=1pt,label=180:{\tiny $x_4$}] (4) at (-.15,1.1) {};
\node [fill=black,circle, inner sep=1pt,label=0:{\tiny $x_3$}] (3) at (1.15,1) {};
\node [fill=black,circle, inner sep=1pt,label=90:{\tiny $x_5$}] (5) at (.5,1.58) {};

\draw [,-stealth ] (1) -- (2) node [ midway,above,  sloped, blue] {}; 
\draw [,-stealth ] (3) -- (5) node [ midway,above,  sloped, blue] {}; 
\draw [,-stealth ] (2) -- (3) node [ midway,above,  sloped, blue] {}; 
\draw [,-stealth ] (2) -- (4) node [ midway,above,  sloped, blue] {}; 
\draw [,-stealth ] (5) -- (1) node [ midway,above,  sloped, blue] {}; 
\draw [,-stealth ] (5) -- (2) node [ midway,above,  sloped, blue] {}; 
\draw [,-stealth ] (4) -- (5) node [ midway,above,  sloped, blue] {}; 

\end{tikzpicture}\label{fig:dhh2}}\\
\subfloat[]{
\begin{tikzpicture}[scale=1.2]
\draw[help lines, step=10pt,dotted](-.2,-.2) grid (2.2,1.9);
\node [fill=black,circle, inner sep=1pt,label=225:{\tiny $x_1$}] (1) at (0,0) {};
\node [fill=black,circle, inner sep=1pt,label=-45:{\tiny $x_2$}] (2) at (2,0) {};
\node [fill=black,circle, inner sep=1pt,label=180:{\tiny $x_4$}] (4) at (0,1) {};
\node [fill=black,circle, inner sep=1pt,label=45:{\tiny $x_3$}] (3) at (1,1.5) {};
\node [fill=black,circle, inner sep=1pt,label=90:{\tiny $x_5$}] (5) at (.5,1) {};

\draw [,-stealth ] (1) -- (2) node [ midway,above,  sloped, blue] {}; 
\draw [,-stealth ] (3) -- (5) node [ midway,above,  sloped, blue] {}; 
\draw [,-stealth ] (2) -- (3) node [ midway,above,  sloped, blue] {}; 
\draw [,-stealth ] (2) -- (4) node [ midway,above,  sloped, blue] {}; 
\draw [,-stealth ] (5) -- (1) node [ midway,above,  sloped, blue] {}; 
\draw [,-stealth ] (5) -- (2) node [ midway,above,  sloped, blue] {}; 
\draw [,-stealth ] (4) -- (5) node [ midway,above,  sloped, blue] {}; 

\end{tikzpicture}\label{fig:dhd1}
}
\subfloat[]{
\begin{tikzpicture}[scale=1.2]
\draw[help lines, step=10pt,dotted](-.2,-.2) grid (2.2,1.9);
\node [fill=black,circle, inner sep=1pt,label=225:{\tiny $x_1$}] (1) at (0,0) {};
\node [fill=black,circle, inner sep=1pt,label=-45:{\tiny $x_2$}] (2) at (2,0) {};
\node [fill=black,circle, inner sep=1pt,label=180:{\tiny $x_4$}] (4) at (0,1) {};
\node [fill=black,circle, inner sep=1pt,label=45:{\tiny $x_3$}] (3) at (1,1.5) {};
\node [fill=black,circle, inner sep=1pt,label=90:{\tiny $x_5$}] (5) at (.5,1) {};

\draw [,stealth-stealth ] (1) -- (2) node [ midway,above,  sloped, blue] {}; 
\draw [,stealth-stealth ] (3) -- (5) node [ midway,above,  sloped, blue] {}; 
\draw [,stealth-stealth ] (2) -- (3) node [ midway,above,  sloped, blue] {}; 
\draw [,stealth-stealth ] (2) -- (4) node [ midway,above,  sloped, blue] {}; 
\draw [,stealth-stealth ] (5) -- (1) node [ midway,above,  sloped, blue] {}; 
\draw [,stealth-stealth ] (5) -- (2) node [ midway,above,  sloped, blue] {}; 
\draw [,stealth-stealth ] (4) -- (5) node [ midway,above,  sloped, blue] {}; 
\draw [,stealth-stealth ] (3) -- (1) node [ midway,above,  sloped, blue] {}; 
\draw [,stealth-stealth ] (4) -- (1) node [ midway,above,  sloped, blue] {}; 
\draw [,stealth-stealth ] (4) -- (3) node [ midway,above,  sloped, blue] {}; 
\end{tikzpicture}\label{fig:dhd2}

}

\end{center}
\caption{We represent in $(a)$ a configuration of 5 points in the plane. Figures $(b)$ and $(c)$ represent two possible $h$-graph for a formation control problem. Figures $(d)$ and $(e)$ two possible $\delta$-graph. Observe that the $\delta$ graph of Figure $(e)$ is fully connected, hence every agent knows the global objective.}
\end{figure}
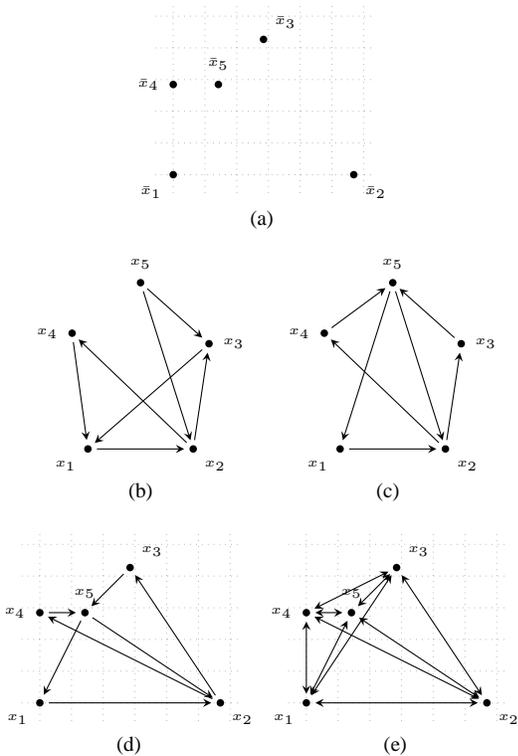

Formation  control problems are concerned with stabilization, either local or type-A. Two different flavors have been studied in the literature: 
\begin{enumerate}

\item {\bf Stabilize at a given framework}:  the global objective is described as the stabilization at the framework described by $\mu$. In this case, the cardinality of $\mathcal E_d$, the set of design equilibria, is one (up to mirror symmetry). Rigidity theory tells us that $G_\delta$ needs to be \emph{globally rigid}~\cite{belabbascdc11sub1}.

\item {\bf Stabilize at one of many framework}: Given a configuration of n-points ($n \geq 4$) $\mu$, let $G_\delta$ be a \emph{minimally rigid graph}. 
The objective is to stabilize at \emph{any} frameworks such that the edge lengths in $G_\delta$ are satisfied. Because the graph is  minimally rigid---and not globally rigid---there are  several frameworks which have the same edge lengths for edges of $G_\delta$ (see Figure~\ref{fig:4formations} for an example, or in Figure~\ref{fig:dhd1}, taking the mirror symmetric of $\bar x_1$ with respect to the $\bar x_2 -\bar x_5$ axis yields a framework with similar edge lengths). In this case, $\mathcal E_d$ is given by all frameworks which satisfy the given edge lengths and the $\delta_i$'s are of the range only type.
\end{enumerate}

We revisit these ideas in Section~\ref{sec:2cycles}. We conclude this section by mentioning  broad open questions in formation control:

\begin{enumerate}
\item How many frameworks satisfy a given set of edge lengths? We have given a lower bound in a particular case in~\cite{belabbascdc11sub1}, but the general case is not settled.
\item How sparse can the graphs $G_\delta$ and $ G_h$  be in order to guarantee the existence of robust $u_i(\delta_i;h_i)$ that yield type-A stabilization?
\item  How sparse can the graphs $G_\delta$ and $G_h$ be in order to guarantee the  existence of robust $u_i(\delta_i;h_i)$ that yield have local stabilization around any point in $\mathcal E_d$?
\end{enumerate}

While rigidity theory  clearly has a role to play in a complete understanding of the $\delta$-graph, it is not clear that it will have more than a supporting role for investigations related to the $h$-graph.

In the case of the $\delta$-graph, a first obvious result is that minimal rigidity (and Laman theorem~\cite{belabbascdc11sub1}) yields a "minimal" undirected $\delta$-graph: a less dense graph is not giving agents enough about the global objective to allow them to satisfy it. The case of directed formation is already much more complex.

We provide partial answers to these question  in Section~\ref{sec:2cycles}.

\section{The two-cycles formation and other known results}\label{sec:2cycles}

\begin{figure}
\begin{center}
\subfloat[The two-cyles formation]{
\begin{tikzpicture}[scale=1.2] 
\node [fill=black,circle, inner sep=1pt,label=-90:$x_1$] (1) at ( 0, 0) {};
\node [fill=black,circle, inner sep=1pt,label=180:$x_2$] (2) at (-1 ,1) {};
\node [fill=black,circle, inner sep=1pt,label=90:$x_3$] (3) at (0 ,2) {};
\node [fill=black,circle, inner sep=1pt,label=0:$x_4$] (4) at (1 ,1) {};

\draw [-stealth ] (1) -- (2);
\draw [-stealth ] (3) -- (1);
\draw [-stealth] (4) -- (3);
\draw [-stealth] (2) -- (3);
\draw [-stealth ] (1) -- (4);
\end{tikzpicture}\label{fig:2cycleshg}}\qquad
\subfloat[The triangle formation]{
\begin{tikzpicture}[scale=2.0] 
\node [fill=black,circle, inner sep=1pt,label=-90:$x_1$] (1) at ( 0, 0) {};
\node [fill=black,circle, inner sep=1pt,label=0:$x_2$] (2) at (1 ,0) {};
\node [fill=black,circle, inner sep=1pt,label=90:$x_3$] (3) at (0 ,1) {};

\draw [-stealth ] (1) -- (2);
\draw [-stealth ] (3) -- (1);
\draw [-stealth] (2) -- (3);
\end{tikzpicture}\label{fig:triangform}}
\end{center}
\caption{}
\end{figure}

We present some known results in formation control  and illustrate in this section the notions introduced in this paper on the 2-cycles formation, which was  exhibited in~\cite{cao2010festschrift} as an example of the difficulty to make progress in formation control when there are "loops of information" in the system. It was conjectured~\cite{cao07cdc} that formation control problems whose objective is minimally rigid, and whose underlying  $\delta$-graph (the $h$-graph was assumed to be the same as the $\delta$-graph) has no vertices with outvalence larger that two were  globally (or type-A)  stabilizable. Since then, we have shown it was not the case for the 2-cycles.

The two-cycles is the formation represented in Figure~\ref{fig:2cycleshg}. Let $x_i \in \R^2$, $i=1\ldots 4$ represent the position of the $4$ agents in the plane. We define the vectors 
\begin{equation}
\label{eq:defz}
\left\lbrace \begin{array}{rcl}
z_1 &=& x_2-x_1 \\
z_2 &=& x_3-x_2 \\
z_3 &=& x_1-x_3\\
z_4 &=& x_3-x_4 \\
z_5 &=& x_4-x_1
\end{array}\right.
\end{equation}Hence, the observation function are given by \begin{multline}\label{eq:defh} h_1(x)=(\|z_1\|,\|z_5\|, z_1^T z_5), h_2(x)=\|z_2\|, h_3(x)=\|z_3\|,\\ h_4(x)=\|z_4\|.\end{multline}

With the notation of Figure~\ref{fig:xbar}, we let $\mu=[\bar x_2, \bar x_3, \bar x_4]$ parametrize a configuration of four points    in the plane. We let $\|\bar x_2\| = d_1$, $\|\bar x_4\|=d_5$, $ \|\bar x_3-\bar x_2\| = d_2$, etc. We take the $\delta$-graph to be the same as the $h$-graph and consider the range only case. Hence,  the functions $\delta_i$ are given by
\begin{equation}
\label{eq:defdel}
\delta_1(\mu) = (d_1,d_5), \delta_2(\mu) =d_2,  \delta_3(\mu) = d_3, \delta_4(\mu) =d_5, 
\end{equation} 

 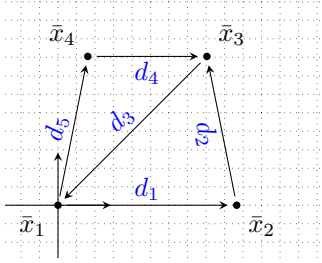
\begin{figure}[ht]\begin{center}
  \begin{tikzpicture}[scale=.7] 
  \draw [-stealth, ] (-1, 0 ) -- (1, 0);
  \draw [-stealth, ] (0,-1  ) -- (0,1);
  \draw[very thin, dotted,step=10pt](-1,-1) grid (5,4);

     \draw[rotate=135,scale=1.6] 	(0,0) node [fill=black,circle, inner sep=1pt,label=-135:$\bar x_1$]  (1) {}
      								(-1.5, -1.5) node [fill=black,circle, inner sep=1pt,label=-45:$\bar x_2$]  (2) {}
      								(0,-2.5) node [fill=black,circle, inner sep=1pt,label=45:$\bar x_3$] (3)  {}
      							(1,-1.5) node [fill=black,circle, inner sep=1pt,label=135:$\bar x_4$] (4)  {};
  
  \draw [-stealth, , ] (1) -- (2) node [ midway,above,  sloped, blue] {$d_1$};
  \draw [-stealth, ] (1) -- (4) node [ midway,above,  sloped, blue] {$d_5$};
  \draw [-stealth, ] (2) -- (3) node [ midway,below,  sloped, blue] {$d_2$};
  \draw [-stealth, ] (4) -- (3) node [ midway,below,  sloped, blue] {$d_4$};
  \draw [-stealth, ] (3) -- (1) node [ midway,above,  sloped, blue] {$d_3$};
  \end{tikzpicture}
  \end{center}
\caption{Any framework in the plane with $\bar x_1 \neq \bar x_2$ is congruent to a framework with $\bar x_1=(0,0) $ and $\bar x_2$ on the $x$-axis}\label{fig:xbar}
\end{figure}

It is convenient to introduce  variables for the error in edge lengths:
$$e_i= z_i^Tz_i-d_i.$$

The set of vector fields that  respect both the invariance of the system under the $SE(2)$ action as presented in~\cite{belabbascdc11sub1} is  given by

Hence a general control law for such a system is
\begin{equation}\label{eq:dyns}
\left\lbrace\begin{array}{rcl}
\dot x_1 &=& u_{11}(\delta_1(\mu);h_1(x)) g_{11}(x) \\ && \mbox{          }+ u_{12}(\delta_1(\mu);h_1(x)) g_{12}(x) \\
\dot x_2 &=& u_2(\delta_2(\mu);h_2(x)) g_2(x) \\
\dot x_3 &=& u_3(\delta_3(\mu);h_3(x)) g_3(x) \\
\dot x_4 &=& u_4(\delta_4(\mu);h_4(x)) g_4(x)
\end{array}\right.
\end{equation}
with \begin{multline*}g_{11}(x)=(x_2-x_1); g_{12}(x)=(x_4-x_1); g_2(x)=x_3-x_2;\\ g_3(x)=x_1-x_3\mbox{ and } g_4(x)=x_3-x_4.\end{multline*} We denote by $\mathcal F$ the space of control systems of the type of Equation~\eqref{eq:dyns}, with the $u_i$ smooth real-valued functions of their argument. We equip $\mathcal F$ with the $C^r$ topology.

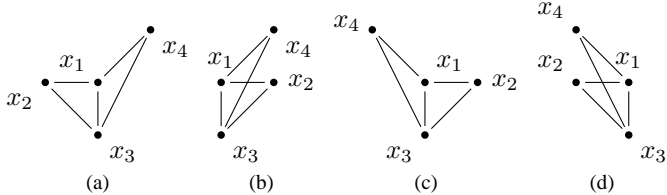
\begin{figure}[ht]
\begin{center}
\subfloat[]{
\begin{tikzpicture}[scale=.7] 
\node [fill=black,circle, inner sep=1pt,label=135:$x_1$] (1) at ( 0, 0) {};
\node [fill=black,circle, inner sep=1pt,label=-135:$x_2$] (2) at (-1 ,0) {};
\node [fill=black,circle, inner sep=1pt,label=-45:$x_3$] (3) at (0 ,-1) {};
\node [fill=black,circle, inner sep=1pt,label=-45:$x_4$] (4) at (1 ,1) {};

\draw [ ] (1) -- (2);
\draw [] (3) -- (1);
\draw [] (4) -- (3);
\draw [] (2) -- (3);
\draw [] (1) -- (4);

\end{tikzpicture}

} 
\subfloat[]{
\begin{tikzpicture}[scale=.7] 
\node [fill=black,circle, inner sep=1pt,label=90:$x_1$] (1) at ( 0, 0) {};
\node [fill=black,circle, inner sep=1pt,label=0:$x_2$] (2) at (1 ,0) {};
\node [fill=black,circle, inner sep=1pt,label=-45:$x_3$] (3) at (0 ,-1) {};
\node [fill=black,circle, inner sep=1pt,label=-45:$x_4$] (4) at (1 ,1) {};

\draw [ ] (1) -- (2);
\draw [] (3) -- (1);
\draw [] (4) -- (3);
\draw [] (2) -- (3);
\draw [] (1) -- (4);

\end{tikzpicture}
} 
\subfloat[]{
\begin{tikzpicture}[scale=.7] 
\node [fill=black,circle, inner sep=1pt,label=45:$x_1$] (1) at ( 0, 0) {};
\node [fill=black,circle, inner sep=1pt,label=0:$x_2$] (2) at (1 ,0) {};
\node [fill=black,circle, inner sep=1pt,label=-135:$x_3$] (3) at (0 ,-1) {};
\node [fill=black,circle, inner sep=1pt,label=180-45:$x_4$] (4) at (-1 ,1) {};

\draw [ ] (1) -- (2);
\draw [] (3) -- (1);
\draw [] (4) -- (3);
\draw [] (2) -- (3);
\draw [] (1) -- (4);

\end{tikzpicture}
} 
\subfloat[]{
\begin{tikzpicture}[scale=.7] 
\node [fill=black,circle, inner sep=1pt,label=90:$x_1$] (1) at ( 0, 0) {};
\node [fill=black,circle, inner sep=1pt,label=135:$x_2$] (2) at (-1 ,0) {};
\node [fill=black,circle, inner sep=1pt,label=-45:$x_3$] (3) at (0 ,-1) {};
\node [fill=black,circle, inner sep=1pt,label=135:$x_4$] (4) at (-1 ,1) {};

\draw [ ] (1) -- (2);
\draw [ ] (3) -- (1);
\draw [] (4) -- (3);
\draw [] (2) -- (3);
\draw [] (1) -- (4);
\end{tikzpicture}
}\end{center}
\caption{\small Four frameworks in the plane that are not equivalent under rotations and translation and that have the same corresponding edge lengths. $(a)$ is the mirror-symmetric of $(c)$ and $(b)$ is the mirror-symmetric of $(d)$.}
\label{fig:4formations}
\end{figure}

The set of design equilibria $\mathcal{E}_d$ for the 2-cycles is of cardinality 4, up to rigid transformations, since there are four frameworks in the plane for which $e_i=0$; they are depicted in Figure~\ref{fig:4formations}. Hence, the global objective\footnote{The global objective of a decentralized system, defined in~\cite{belabbas_icca2011_knowns}, is achieved at configurations $x$ such that $F(\mu;x)=0$.} can be written as an equality objective with $$F(\mu;x) = \left[\begin{matrix} \|x_2-x_1\|^2-d_1 \\
\|x_3-x_2\|^2-d_2 \\
\|x_1-x_3\|^2-d_3 \\
\|x_3-x_4\|^2-d_4 \\
\|x_1-x_4\|^2-d_5
\end{matrix}\right]$$ with the additional requirement of either stabilizing locally any of these equilibria or seeking a control such that the system is type-A stable.

The local objectives\footnote{Local objectives, also defined in~\cite{belabbas_icca2011_knowns}, are achieved by an agent at configurations such that $f_i(\delta_i(\mu);h_i(x)) = 0$.} for each agents are to stabilize at the required distance from their neighbors. For agent 1, we have $$f_1(\delta_1(\mu);h_1(x)) = \left[\begin{matrix} 
\|x_2-x_1\|^2-d_1 \\
\|x_4-x_2\|^2-d_5
\end{matrix}
\right],$$ and for agents $i=2,3,4$:
\begin{eqnarray*}
f_2(\delta_2(\mu);h_2(x)) &=& \|x_3-x_2\|^2-d_2\\
f_3(\delta_3(\mu);h_3(x)) &=& \|x_1-x_3\|^2-d_3\\
f_4(\delta_4(\mu);h_4(x)) &=& \|x_3-x_4\|^2-d_4
\end{eqnarray*}
Satisfying the local objectives clearly implies that the global objective is satisfied.

In general,  the set $\mathcal{E}_a$ of ancillary equilibria depends  on the choice of feedbacks $u_i$. Due to the invariance and distributed nature of the system, we can exhibit some configurations that belong to $\mathcal{E}_a$ for all $u_i$'s.~\cite{belabbasSICOpart1}:

\begin{Proposition}\label{prop:allequ} The set $\mathcal{E}$ contains, in addition to the equilibria in $\mathcal{E}_d$,  the  frameworks characterized by
\begin{enumerate}
\item $z_i=0$ for all $i$,  which corresponds to having all the agents superposed.
\item all $z_i$ are aligned, which corresponds to having all agents on the same one-dimensional subspace in $\R^2$. These frameworks form a three dimensional invariant subspace of the dynamics.
\item  $e_2=e_3=e_4=0$, $z_1$ and $z_5$ are aligned and so that $$u_1(\delta_1;h_1) \|z_1\|=\pm u_5(\delta_1;h_1) \|z_5\|,$$ where the sign depends on whether $z_1$ and $z_5$ point in the same or opposite directions. 
\end{enumerate}
\end{Proposition}

As we have discussed in the companion paper~\cite{belabbas_icca2011_knowns}, the $\delta_i$ given in Equation~\eqref{eq:defdel} do not saturate  the observation functions. Hence there may be some gain in letting $\delta$ be more informative. We know that a maximally informative $\delta_i$ would be given by the identity function. This maximally informative $\delta_i$ was used in~\cite{yu09} to prove that the 2-cycles can be locally stabilized at a given framework in the plane using a relatively simple control law and adjusting some feedback gains. The dynamics used was of the type of Equation~\ref{eq:dyns} with the control law   \begin{equation}\label{eq:controlyu}u_i = k_i e_i, i=1,\ldots,5\end{equation} where the $k_i$ are constant real-valued gains used by the agents to locally stabilize a given framework (i.e. $k_i=k_i(\mu)$).

We can restate the theorem in the language of this paper as follows: let $E^4$  be the space of all configurations of  $4$ points in the plane.

\begin{Theorem}[Reformulation of~\cite{yu09}]\label{th:anderson}
If we let $\delta_i(\mu)=\mu$ for all $i$ and $h_i$ as in Equation~\eqref{eq:defh}, there exists $u_i(\mu; h_i(x))$ such that  for all $\mu$, the framework parametrized by $\mu$ is locally stable for the system  of Equation~\eqref{eq:dyns} with controls $u_i$. In fact, a  control law  of the type of Equation~\eqref{eq:controlyu} works for configurations of $n$ points in $\R^n$, with $h$-graph  given by a minimally rigid graph with outvalence at each node at most two (and  $\delta_i$ being maximally informative).
\end{Theorem}
\begin{proof}[Sketch of proof]
The proof relies on the linearization of the system around a given framework. It is then showed that by multiplying the Jacobian of the system by a block diagonal matrix---corresponding to the gains---one can make all eigenvalues of the product have negative real part. To this end, a result similar to the one in~\cite{friedland85} is proved for the case of real matrices.
\end{proof}

We  have shown in~\cite{belabbasSICOpart2} that the same result does not hold if  we let the $\delta$-graph be the same as the $h$-graph:

\begin{Theorem}\label{th:localstab}
Given $\delta_i(\mu)$ as in Equation~\eqref{eq:defdel} and $h_i$ as in Equation~\eqref{eq:defh}, there are no \emph{robust} control system in $\mathcal{F}$  that locally stabilize all frameworks in $E^4$. In fact, for any $u_i \in \mathcal U_i$, there exists a set of frameworks of positive measure in $E^4$ that are not locally stabilizable.
\end{Theorem}
\begin{proof}[Sketch of proof]
The proof relies on showing that, given the $\delta_i$ and $h_i$, satisfying the local objectives robustly prevents the global stabilization objective to be satisfied.
\end{proof}

Local stabilization of formations with either  symmetric or  cycle-free $h$- and $\delta$-graph, is much easier to handle. We mention here that \emph{linear decentralized control problem} whose information flow was given by a graph without cycles have been studied in~\cite{shah09}.  We cite the following result from formation control, which relies on similar graphs in the directed case:

\begin{Theorem}[Local stabilization of bi-directional formations~\cite{krick08}] Given a configuration of $n$ points in the plane with $\delta$-graph $G=(V,E)$ that is  infinitesimally rigid , and $h$-graph equal to $\delta$-graph, the control law
$$u_i(\delta_i(\mu), h_i(x))= \sum_{j \mbox{ s.t. } (i,j) \in E} ( \|x_i-x_j\|-d_k ) (x_i-x_j) $$ locally stabilizes almost all infinitesimally rigid frameworks. The same holds true for \emph{directed formations}, where the $\delta$- and $h$- graphs are the same and contain \emph{no loops} and every vertex has outvalence of two at the most.
\end{Theorem}
\begin{proof}[Sketch of proof]
The proof is  based on a linearization of the system about an equilibrium.
\end{proof}
For a description of the configurations that are excluded by the qualifier  "almost all", we refer the reader to~\cite{krick08}.

Results about global stabilization of formations, whether directed or undirected, are much more sparse. We mention the result of~\cite{anderson07} about the triangular formation:

\begin{Theorem}[Type-A stability of triangular formation.~\cite{anderson07}] Consider the triangular formation of Figure~\ref{fig:triangform}. The control law $$\dot x_i = (\|x_{i+1}-x_i\|-d_i) (x_{i+1}-x_i)$$ makes the system robustly type-A stable for almost all configurations of points in $E^3$
\end{Theorem}
\begin{proof}[Sketch of proof]
The proof is based on exhibiting a Lyapunov-like~\cite{cao2010festschrift} function for the system and showing that, except for what is called a "thin set" of initial conditions (i.e. a set of codimension one), the system is globally stable. This can be rephrased using the type-A stability idea as we have done here.
\end{proof}

Even more, type-A stability was shown, using a similar Lyapunov argument, for a broad class of decentralized control law in~\cite{cao2010festschrift}. We have shown that this result does not extend to the 2-cycles:

\begin{Theorem}~\label{th:globstab} There are no robust $u \in \mathcal U$ with $\delta_i$ as in Equation~\eqref{eq:defdel} and $h_i$ as in Equation~\eqref{eq:defh} such that the 2-cycles formation is type-A stable.
\end{Theorem}
\begin{proof}[Sketch of proof]
The theory of bifurcation and singularities was used to show that the decentralized structure of the system forces the appearance of stable, ancillary equilibria for all feedback laws $u \in \mathcal U$.
\end{proof}

Whether letting $\delta_i$ to be the identity would allow to find a robust, type-A stabilizing control for the two-cycles is an open question. 
\section{Summary}

We have defined formation control problems in the plane and introduced some relevant mathematical concepts: type-A stability and robustness. The presentation of formation control highlighted the difference between the decentralization of the system as it is commonly understood (agents have a partial information about the state of the ensemble) and the decentralization of the objective (agents have a partial information about what configuration the formation is asked to reach). We have seen that the latter type of decentralization, though not often acknowledged,  affects the behavior of formations greatly (compare Theorems~\ref{th:anderson} and~\ref{th:localstab}). Finally, we have presented some open questions in formation control.

\appendix
\section{Singularities of vector fields, jet spaces and transversality}\label{app:appsing}

The main tool for handling  genericity and robustness in function spaces is Thom's transversality theorem. We will arrive at the result by building onto the simpler concept of transversality of linear subspaces.

Thom's theorem roughly answers the following type of questions: given  a function $u$ from a manifold $M$ to a manifold $N$, and some relations between the derivatives of different orders of this function (e.g. $u''+u'-u=0$), under what circumstances are these relations preserved under small perturbations of the function? For example, if a real-valued function has a zero at some point, under a small perturbation of this function, the zero will persist \emph{generically} for $u$. On the other hand, if a real-valued function vanishes with its second derivative also being zero, under a small perturbation  this property will be lost, see Figure~\ref{fig:illpertxsq}. The crux of Thom's theorem is to show that considering only a ``small subset'' of perturbations (the integrable perturbations as we will see below) in the set of all perturbations in jet-spaces is sufficient. 

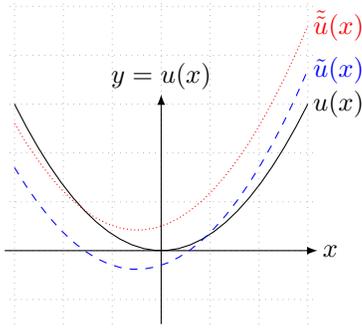
\begin{figure}[]
\begin{center}
\begin{tikzpicture}[scale=.65,domain=-3:3] 
\draw[help lines, dotted,color=gray] (-3.1,-1.5) grid (3.1,5.1);
\draw[->] (-3.2,0) -- (3.2,0) node[right] {$x$}; 
\draw[->] (0,-1.5) -- (0,3.2) node[above] {$y=u(x)$};
\draw[color=black]	plot (\x,1/3*\x*\x)	node[right] {$u(x)$}; 
\draw[color=blue,dashed]	plot (\x,1/3*\x*\x+1/3*\x-.3)	node[right] {$\tilde u(x) $}; 
\draw[color=red,densely dotted] plot (\x,{1/2.9*\x*\x+1/3*\x+.5}) node[right] {$\tilde{\tilde{u}}(x)$};
\end{tikzpicture}
\caption{If we let $\mathcal P$ be the property of vanishing with a zero derivative. We will prove in this section that $\rceil \mathcal P$ is generic and thus  $\mathcal P$  is not robust. Let $u(x)$ be a function which satisfy $\mathcal P$. For any small perturbations, it will either vanish with a non-zero derivative---as illustrated with $\tilde u(x)$, dashed curve--- or not vanish at all---as illustrated with $\tilde{\tilde{u}}(x) $, dotted curve. Both $\tilde u(x)$ and $\tilde{\tilde{u}}(x)$ are transversal to the manifold defined by $y=0$ everywhere, whereas $u(x)$ is not.}\label{fig:illpertxsq}
\end{center}
\end{figure}

Let  $A, B \subset \R^n$ be linear subspaces. They are \emph{transversal} if $$\R^n = A \oplus B, $$ where $\oplus$ denotes the direct sum. For example, a plane and a line not contained in the plane are transversal in $\R^3$. The notion of transversality can be extended to maps as follows: given $$F_1 : \R^n \rightarrow \R^m\mbox{ and } F_2: \R^l \rightarrow \R^m,$$ we say that $F_1$ and $F_2$ are transversal at a point $(x_1,x_2) \in \R^n \times \R^l$ if one of the two following conditions is met:

\begin{enumerate}
\item $F_1(x_1) \neq F_2(x_2)$
\item If $F_1(x_1) = F_2(x_2)$, then the matrix $\left[ \begin{matrix}\frac{\partial F_1}{\partial x} \\\frac{\partial F_2}{\partial x} \end{matrix}\right]$ is of full rank.
\end{enumerate}
In particular, if $l+n < m$ then $F_1$ and $F_2$ are transversal only where they do not map to the same point. This definition extends immediately to smooth functions between smooth manifolds: given $$F_1:M_1\rightarrow N \mbox{  and } F_2:M_2 \rightarrow N,$$ we say that $F_1$ and $F_2$ are transversal at $(x_1,x_2)\in M_1\times M_2$ if either $F_1(x_1)\neq F_2(x_2)$ or $F(x_1)=F(x_2)$ and the tangent space of $N$ at $F(x_1)$ is the direct sum of the images of the tangent spaces of $M_1$ and $M_2$ under $F_1$ and $F_2$ respectively, i.e. $$T_{F_1(x_1)}N = F_{1*}T_{x_1}M_1 \oplus F_{2*}T_{x_2}M_2$$ where $F_*$ is the push-forward~\cite{warner83} of $F$. 

\begin{Example}
Take $M_1=\R$ and $M_2=N=\R^2$ and let $F_1(x_1)=x_1v+b_1$ and $F_2(x_2)=Ax_2+b_2$, where $A \in \R^{2 \times 2}, b_2, v,x_2 \in \R^2$ and $x \in \R$. If $b_1\neq b_2$, then $F_1(0) \neq F_2(0)$ and $F_1$ is transversal to $F_2$ at $(0,0,0)$. If $b_1=b_2$, then $F_1(0)=F_2(0)$	and the functions are transversal if the span of $v$ and the columns of $A$ is $\R^2$.
\end{Example}

The notion of transversality that is of interest to us is a straightforward extension of the transversality of maps:

\begin{Definition}[Transversality] \label{def:transversality} Let $F:M \rightarrow N$ be a smooth map and let $C$ be a submanifold of $N$. Then $F$ is \emph{transversal} to $C$ at a given point if, at that point,  $F$ is transversal to the embedding $i:C \rightarrow N$ of $C$ into $N$.
\end{Definition}

\begin{Example}
Take $N=\R^3$ with coordinates $u,v,w$ and $C$ be the u-v plane. Let $F:\R \rightarrow \R^3:x \rightarrow [x,2x,3x]^T$. Then the map $F$ is transversal to $C$ everywhere.
\end{Example}

\begin{Example}
Let $F(x):\R \rightarrow \R^3$ be any smooth curve in $\R^3$ and $C$ be the $u$-axis. At points where $F(x) \in C$, the tangent vector to $C$ and the tangent vector to $F$ will span at most a two-dimensional subspace in $\R^3$.  Hence, $F$ is transversal to $C$ only at the points where $F(x) \notin C$.
\end{Example}

\subsection{Jet Spaces}

Let $F,G :M \rightarrow N$ be  smooth maps between smooth manifolds $M$ and $N$ endowed with a metric. We say that $F$ and $G$ are k- equivalent at $x_0 \in M$ if in a neighborhood of $x_0$ we have $$\|F(x)-G(x)\| = o(\|x-x_0\|^k).$$  One can verify~\cite{arnold72} that k-equivalence is independent of the choice of metrics on $M$ and $N$ and that it is an equivalence relation on maps. In fact, the above definition can be recast as saying that  $F$ and $G$ are  0-equivalent at $x_0$ if $$F(x_0)=G(x_0),$$  1-equivalent if in addition $$\frac{\partial F}{\partial x}|_{x_0}=\frac{\partial G}{\partial x}|_{x_0},$$ and so forth. We define the k-jet of a smooth map to be its k-equivalence class:

\begin{Definition}  
The k-jet of $F:M \rightarrow N$ at $x_0$ is $$J_{x_0}^k(F) = \lbrace G:M \rightarrow N\mbox{ s.t. } G \mbox{ is }\mbox{k-equivalent to } F \rbrace .$$
\end{Definition}
Hence, the 0-jet of $F$ at $x_0$ is $F(x_0)$; the 1-jet is $(F(x_0), \frac{\partial F}{\partial x}|_{x_0})$, etc.  For example, the constant function $0$ and $\sin(x)$ have the same 0-jet at $x=0$ and $x$ and $\sin(x)$ have the same 1-jet at $0$. 

We define: $$J^k(M,N) = \mbox{ Space of k- jets from }M\mbox{ to } N .$$

A 0-jet is thus determined by a point in $M$ and a point in $N$, and thus $J^0(M,N)$ nothing more than the Cartesian product of $M$ and $N$: $$J^0(M,N) = M \times N.$$ Since a 1-jet is determined by a pair of points,for the 0-jet part, and a matrix of dimension $\dim M \times \dim N $, for the Jacobian of the function at $x_0$, we see that $\dim J^1(M,N)=\dim M+\dim N + \dim M \dim N$. We cannot say in general that $J^1(M,N)$ is the cartesian product of $J^0$ with $\R^{m\times n}$ since the product may be twisted. Similar relations are obtained for higher jet-spaces~\cite{arnold72}

Given a function $F:M \rightarrow N$, we call its \emph{k-jet extension} the map given by
$$j^k_F(x):M \rightarrow J^K(M,N): x \rightarrow (F(x), \frac{\partial F}{\partial x}(x),\ldots, \frac{\partial^k F}{\partial x^k}(x)).$$

\begin{Example} Let $M=N=\R$. The k-jet space is  $J^k(\R,\R)= \R \times \R \times \ldots \times \R = \R^{k+2}$. Take $F(x)=\sin(x)$; the $2-$jet extension of $F$ is $$j^2_{\sin}(x) = (x,\sin(x), \cos(x),-\sin(x)).$$
If we take $M=N=\R^2$ and $F(x)=Ax$ for $A \in \R^{2 \times 2}$, then $$j_{Ax}^k(x)=(x,Ax, A, 0, \ldots, 0).$$
\end{Example}

\begin{Remark}
The concepts presented in this section also trivially apply to vector fields on $M$, by letting $N=TM $.
\end{Remark}
While to any function $F:M \rightarrow N$, we can assign a k-jet extension $j_F^k:M \rightarrow J^k(M,N)$, the inverse is not true: there are maps $G:M \rightarrow J^k(M,N)$ which do not correspond to functions from $M$ to $N$ as there are some obvious integrability conditions that need to be satisfied. For example, if we let $$G:\R^n \rightarrow J^1(\R^n): G(x)=(x,Ax,B),$$ then  $G$ is a 1-jet extension of a function if and only if $B = A$.

The power of the transversality theorem of Thom is that it allows one  to draw conclusions about transversality properties in general, and genericity in particular, by \emph{solely looking at perturbations in jet spaces that are jet extensions}---a much smaller set than  all perturbations in jet-spaces, since these include the much larger set of non-integrable perturbations.  

We recall that the $C^r$ topology is a metric topology. It is induced by a metric that takes into account the function and its first $r$ derivatives. We have:

\begin{Theorem}[Thom's transversality]\label{th:thom} Let $C$ be a  regular submanifold of the jet space
$J^k(M,N)$. Then the set of maps $f : M  \rightarrow  N$ whose k-jet extensions are transversal to $C$
is an everywhere dense intersection of open sets in the space of smooth maps for the $C^r$ topology, $1 \leq r \leq \infty$.
\end{Theorem}

A typical application of the theorem is to prove that vector fields with degenerate zeros are not generic. We here prove a version of this result that is of interest to us. 

\begin{Corollary}\label{cor:corgent}
Functions in $\mathcal{C}^\infty(M) $ whose derivative at a zero vanish are not generic.
\end{Corollary}

In other words, the corollary deals with the intuitive fact that if $u(x)=0$, then generically $u'(x)\neq 0$. This result also goes under the name of weak-transversality theorem~\cite{guillemin74}.
\begin{proof}
 Consider the space of 0-jets $J^0(M,\R)$. In this space, let $C$ be the set of 0-jets which vanish, i.e. $C=(x,0) \subset J^0$. A function $u$ is transversal to this set if either it does not vanish, or where it vanishes we have that  the matrix $$\left[ \begin{matrix} 1 & 1 \\ 0 & \frac{\partial f}{\partial x}\end{matrix}\right] $$ is of full rank. Hence, transversality to $C$ at a zero implies that the derivative of the function is non-zero. The result is thus a consequence of Theorem~\ref{th:thom}. 
\end{proof}

To picture the situation geometrically, recall that $J^0(M,\R)$ is simply $M \times \R$.  Hence $C$ is $M \times 0 \subset J^0(M,\R)$. The result says that any function that intersects $C$ without crossing (and hence with a zero derivative)  will, under a generic perturbation, either cross $C$  or not intersect $C$ at all, since these two eventualities result in transversality. Figure~\ref{fig:illpertxsq} provides an illustration when $M= \R$.

\bibliographystyle{IEEEtran}
\bibliography{distribcontrolbib}             

\end{document}